\documentclass[11pt,a4paper]{article}
 
\usepackage[utf8]{inputenc}
 
\usepackage{array,amsbsy,amscd,amsfonts,color,amssymb,amstext,amsmath,latexsym,amsthm,amscd
,multicol,graphicx}
\usepackage{hyperref}
 
\newtheorem{theorem}{Theorem}

\newtheorem{lemma}[theorem]{Lemma}

\theoremstyle{definition}
\theoremstyle{definition}
\theoremstyle{definition}
\theoremstyle{definition}
\theoremstyle{definition}
\theoremstyle{definition}
\theoremstyle{definition}
 
\def\proofof [#1] {\noindent {\bf Proof of #1. } }

\def\al #1.{{\mathcal{#1}}}

\newcommand{\KK}{\mathfrak{K}_\A}

\newcommand{\A}{\mathcal{A}}
\newcommand{\Ring}{\mathcal{R_A}}
\newcommand{\RRing}{\mathcal{\tilde{R}_A}}

\newcommand{\K}{\mathcal{K}}
\renewcommand{\H}{\mathcal{H}}
\newcommand{\M}{\mathcal{M}}

\newcommand{\C}{\mathbb{C}}

\renewcommand{\S}{S^1}

\newcommand{\Z}{\mathbb{Z}}

\newcommand{\E}{\mathcal{E}}
\newcommand{\unit}{\mathbf{1}}

\newcommand{\bp}{\begin{proof}}
\newcommand{\ep}{\end{proof}}
\newcommand{\bdp}{\begin{dproof}}
\newcommand{\edp}{\end{dproof}}
\newcommand{\ra}{\rightarrow}

\newcommand{\Mat}{\operatorname{M}}

\newcommand{\locn}{\operatorname{ln}}

\newcommand{\End}{\operatorname{End}}

\newcommand{\id}{\operatorname{id}}

\newcommand{\ie}{{i.e.,\/}\ }

\newcommand{\eg}{{e.g.\/}\ }
\newcommand{\cf}{{cf.\/}\ }

\title{\huge Conformal Nets and KK-Theory}

\author{
\phantom{X}\\
{\sc Sebastiano Carpi}$^{1}$\footnote{Supported in part by the ERC
Advanced Grant 227458 "Operator Algebras and Conformal Field Theory"},
{\sc Roberto Conti}$^2$,
{\sc Robin Hillier}$^{3*}$\\
\phantom{X}\\
${}^1$ Dipartimento di Economia,
Universit\`a di Chieti-Pescara ``G. d'Annunzio''\\
Viale Pindaro, 42, I-65127 Pescara, Italy\\
E-mail: {\tt s.carpi@unich.it}\\
\phantom{X}\\
${}^2$
SBAI, Sezione di Matematica,
Universit\`a Sapienza di Roma\\
Via A. Scarpa, 1, I-00161 Roma, Italy\\
E-mail: {\tt roberto.conti@sbai.uniroma1.it}\\
\phantom{X}\\
${}^3$
Dipartimento di Matematica,
Universit\`a di Roma ``Tor Vergata''\\
Via della Ricerca Scientifica, 1, I-00133 Roma, Italy\\
E-mail: {\tt hillier@mat.uniroma2.it}
}
 
\date{}

\begin{document}
\maketitle

\begin{abstract} 
Given a completely rational conformal net $\A$ on $S^1$, its fusion ring acts faithfully on the K-group $K_0(\KK)$ of a certain universal C*-algebra $\KK$ associated to $\A$, as shown in a previous paper. We prove here that this action can actually be identified with a Kasparov product, thus paving the way for a fruitful interplay between conformal field theory and KK-theory.
\end{abstract}

\section{Introduction}

In the operator algebraic approach a chiral conformal fied theory is described by a conformal net $\A$ on $\S$, namely an inclusion preserving map $I \mapsto \A(I)$ from the set of proper (nonempty, nondense, open) intervals of the unit circle $\S$ into the family of von Neumann algebras (actually type III$_1$ factors) acting on a fixed separable (complex) Hilbert space 
$\H$ called the {\it vacuum Hilbert space} of the theory, which is conformally covariant and satisfies some other natural conditions, see e.g. \cite{KL04} and the references therein. The superselection structure of the theory is then captured by the representation theory of the net $\A$ and can also be described in terms of localized (DHR) endomorphisms.

In a recent paper \cite{CCHW12} together with Mihály Weiner we have started an investigation of K-theoretical 
aspects of conformal nets. In the present note we take one further step 
along these lines
and provide an interpretation of some of the results obtained 
there in the framework of Kasparov KK-theory. 
At a first glance, this might appear merely as an academic curiosity. However, in view of the striking success of KK-theory on the one hand (\eg \cite{CMR} for an overview) and the
noncommutative geometrization program for (super-) conformal nets and their representations \cite{Lo01,CKL08,CHKL10} on the other hand, we feel that there are good reasons to give a closer look at 
this subject, as it could possibly reveal a lot of potential 
for further investigations.

We recall the results in \cite{CCHW12} relevant for the present analysis,
referring the reader to that paper for more details, notation, and references.
Let $\A$ be a completely rational local conformal net on $S^1$. We can then define the so-called locally normal universal C*-algebra $C^*_{\locn}(\A)$ which is canonically associated to the superselection theory of the net. It may be expressed as a quotient of the universal algebra $C^*(\A)$  previously defined by Fredenhagen in \cite{Fre90} or, equivalently, as the image of $C^*(\A)$ in the so-called locally normal universal representation $(\pi_{\locn},\H_{\locn})$. The striking fact is that $C^*_{\locn}(\A)$ turns out to be isomorphic to the direct sum $B(\H)^{\oplus n}$ of $n$ copies of the algebra $B(\H)$ of all bounded operators on the vacuum Hilbert space of $\A$. Here, $n < \infty$ is the number of sectors of $\A$. Consequently, it has trivial K-theory and several other properties which make it not really attractive as a C*-algebra. 

We therefore consider its separable C*-subalgebra $\KK$ generated by 
the finite projections. 
It is isomorphic to $\K^{\oplus n}\subset B(\H)^{\oplus n}$, where 
$\K := K(\H)$ is the algebra of compact operators on $\H$. 
Clearly $\KK$ is an ideal of $C^*_{\locn}(\A)$ and in fact it is the 
largest among the norm-separable ideals. 
We have
\[
K_0(\KK)=\Z^n,\quad K_1(\KK)=0.
\] 
Furthermore, we get a faithful semiring action of the fusion semiring $\Ring$ of $\A$ on $K_0(\KK)$, 
via restriction of endomorphisms of $C^*_{\locn}(\A)$ to $\KK$ and subsequent pushforward to $K_0(\KK)$.
Namely,  as shown in  from \cite[Th.4.4]{CCHW12}, there is an injective semiring homomorphism
\[
\eta:\Ring\ra \End(K_0(\KK)) 
\]
satisfying $\eta_{[\rho]} = (\hat{\rho}|_{\KK})_*$ for every 
localized covariant endomorphism $\rho$ of $C^*(\A)$, where $[\rho]\in\Ring$ is the sector (\ie the unitary equivalence class) represented by $\rho$, and $\hat\rho$ is a normal *-homomorphism of $C^*_{\locn}(\A)$ naturally associated to $\rho$, mapping $\KK$ into itself. Hereafter, we shall use ``$[\rho]$`` exclusively to denote the sector determined by $\rho$ and not its KK-class.
 
We shall give an interpretation of this action in terms of KK-theory. This fact can be seen as an illustration in the conformal nets setting of  a comment of Izumi \cite[page 118]{Izumi02},  based on previous ideas of Kajiwara and Watatani \cite{KW00,Wat90}, on the relation between sector theory and  KK-theory. There they work out the most natural way of associating a KK-class to a *-homomorphism between two simple C*-algebras. In our context, however, it is a priori unclear what are the correct C*-algebras to be considered and whether we can get rid of their assumptions. In \cite{EG,EG2}, Evans and Gannon discuss KK-theoretical aspects of CFT in relation to modular invariants and twisted K-theory. Following some of their ideas, we shall provide a KK-theoretical interpretation of modular invariants in our setting towards the end of the present article.

Another instance of the emergence of KK-theoretical concepts in AQFT has been recently pointed out by Conti and Morsella \cite{CM12}, where the DHR sectors of scaling limit nets (as defined by Buchholz-Verch, in order to cast the renormalization group analysis into operator algebraic terms) 
are described by maps of the original global quasi-local C*-algebra on four-dimensional Minkowski spacetime  that are suitable modifications of the asymptotic morphisms of Connes and Higson.
 
We close this introduction by mentioning that
other relationships of seemingly quite different nature between KK-theory and quantum field theory have been investigated by other authors, see e.g. \cite{BMRS08}.

\section{Basics of KK-theory}
 
It is well-known that the push-forward of *-homomorphisms gives rise to elements in KK-theory. 
We want to see how this works in our setting in relation to the semiring action $\eta$ recalled above. 
For the convenience of the reader and in order to fix the notation let us introduce here KK-theory very briefly, while referring to \cite{Bl} and \cite{JTh} for details. 
All C*-algebras we consider below will be separable and stably unital. 
In particular, the various approaches to KK-theory become all equivalent \cite[Sec.17\&18]{Bl} and $K_0(A)$ is the Grothendieck group of the projection semigroup of the C*-algebra $\K\otimes A$.
Moreover, we will restrict ourselves to the ungraded algebra case.
 
A \emph{Kasparov $(A,B)$-module} is a tuple $(\E,\phi,F)$, where $\E$ is a countably generated graded Hilbert $B$-module, $\phi:A\ra B(\E)$ is a graded *-homomorphism, and $F\in B(\E)$ has degree one, such that
\[
  (F-F^*)\phi(a), \quad (F^2-\unit) \phi(a), \quad [F,\phi(a)]
\]
lie all in $K(\E)$, the compact operators on $\E$, for all $a\in A$. 
With the usual concept of homotopy, one defines $KK(A,B)$ as the set of homotopy equivalence classes of Kasparov $(A,B)$-modules. Moreover, if we have two Kasparov $(A,B)$-modules $(\E_i,\phi_i,F_i)$ and a unitary in $B(\E_1,\E_2)$ intertwining the $\phi_i$ and the $F_i$, then the two Kasparov modules are homotopically equivalent. There is a direct sum for Kasparov $(A,B)$-modules, which passes to the quotient $KK(A,B)$; \cf \cite[Sec.17]{Bl} for all these statements.
 
For our immediate purposes, the most relevant facts about KK-theory can be summarized as follows:
\begin{theorem}
\label{th:KK-main}
Let $A,B,C$ be C*-algebras as above.
\begin{itemize}
\item[$(1)$] $KK(A,B)$ is an abelian group with respect to the above addition, and $KK$ is a bifunctor from the category of C*-algebras to abelian groups \cite[17.3\&17.8]{Bl}.
\item[$(2)$] There is a canonical identification of $KK(\C,A)$ with $K_0(A)$ (as additive groups), \cite[17.5.5]{Bl}. 
\item[$(3)$] Every *-homomorphism $\phi:A\ra B$ naturally defines a $KK(A,B)$-element $\{\phi\}$
as the homotopy class of $(B,\phi,0)$, where we have identified $B(B)$ with 
the multiplier algebra $\M(B)$.
\item[$(4)$] If two *-homomorphisms $\phi,\psi:A\ra B$, are unitarily equivalent in  
$\M(B)$, then the induced $KK(A,B)$-elements coincide,
\cite[Sec.1.3]{JTh}.
\item[$(5)$] There exists a bilinear map $\times$, the so-called Kasparov product 
\[
  KK(A,B)\times KK(B,C) \ra KK(A,C),
\]
which is associative, \cite[Sec.18]{Bl}.
\item[$(6)$] If $\psi : A \to B$ and $\phi : B \to C$ are *-homomorphisms then 
$$\{\psi\} \times \{\phi\} = \{\phi\circ\psi\}. $$
If  $\;\id_A$  is the identity automorphism of $A$ then $\{\id_A\}$ is the neutral element in $KK(A,A)$ for the Kasparov product.
Hence $KK(A,A)$ is a unital ring. 
\item[$(7)$] The adjoint of the pairing $KK(\C,A) \times KK(A,B) \to KK(\C,B)$ defines,
  via the identification in (2), a map
  $\gamma: KK(A,B) \to {\rm Hom}(K_0(A),K_0(B))$, namely
  \[
  \gamma(y)(x) = x\times y,\quad  x\in K_0(A), y\in KK(A,B).
  \]
  For $y=\{\phi\}$ with $\phi:A\ra B$ a *-homomorphism, we have
  \[
  \gamma(\{\phi\})=\phi_*,
  \]
  the push-forward of $\phi$ in $K$-theory.
\end{itemize}
\end{theorem}
 
We only mention that property (6) may be obtained as a consequence of the functoriality properties of the KK-product \cite[18.7.1]{Bl} and the relation $\{\phi\} =\phi^*\{\id_C\}$.
The canonical identification in (2) is given by
\[
[p_+]-[p_-] \in K_0(A) \mapsto \Big[ \H_A\oplus \H_A, \phi_{p_+} \oplus \phi_{p_-}, 
\begin{pmatrix} 0 & \unit \\ \unit & 0 \end{pmatrix}\Big] \in KK(\C,A),
\]
where $\H_A$ is the Hilbert module $A\otimes l^2(\Z)$, the grading on $\H_A\oplus\H_A$ is $\unit\oplus-\unit$, $p_\pm \in A\otimes \K$ are projections, and $\phi_{p_\pm}$ is the map: $t\in \C\mapsto t p_\pm \in \M(A\otimes\K)$. It has, in particular, the property that push-forwards in $K_0(A)$ and $KK(\C,A)$ mutually correspond to each other. All this can be easiest seen as outlined in  \cite[Sec.1]{Cu} following Cuntz' quasi-homomorphism picture; alternatively one may use \cite[17.5.5]{Bl}. Finally, property (7) is a consequence of (2) and the functoriality property \cite[17.8.2]{Bl}.

\section{Main result}
 
Let us return to our setting, using the notation from the theorem and Section 4 
in \cite{CCHW12}.
 
\begin{theorem}\label{th:KK-sectors} Let $\A$ be a completely rational net on $S^1$, 
and let $\eta:\Ring\ra \End(K_0(\KK)) $ be the injective semiring homomorphism recalled above. 
Then there is an injective unital semiring homomorphism $j:\Ring \ra KK(\KK,\KK)$ such that
\begin{equation}\label{rel}
x \times j([\rho]) = \eta_{[\rho]}(x), \quad [\rho]\in\Ring, \ x\in K_0(\KK)=KK(\C,\KK).
\end{equation}
\end{theorem}
 
\begin{proof}
Using Theorem \ref{th:KK-main}(1),
define 
\[
j([\rho]) := \{\hat{\rho} |_{\KK}\} \in KK(\KK,\KK), [\rho] \in \Ring \ .
\] 
This map is well-defined since if $\rho$ and $\rho'$ are equivalent endomorphisms in $C^*(\A)$ then $\hat{\rho}$ and $\hat{\rho'}$ are equivalent endomorphisms of $C^*_{\locn}(\A)$. As the latter contains $\KK$ as an ideal and it is weakly closed in $B(\H_{\locn})$, it can be easily identified with the multiplier algebra $\M(\KK)$. Then by Theorem \ref{th:KK-main}(2), $\hat{\rho} |_{\KK}$
and $\hat{\rho'} |_{\KK}$ give rise to Kasparov $(\KK,\KK)$-modules, which according to (3) are equivalent and so define the same KK-element $\{\hat{\rho} |_{\KK}\}$.
 
Furthermore, $j$ is multiplicative, namely $j([\rho_1][\rho_2]) = j([\rho_1]) \times j([\rho_2])$ 
and $j([\id])$ is the unit in $KK(\KK,\KK)$. Indeed, using the commutativity of the composition of DHR sectors 
(due to the existence of a unitary braiding) and  point (6) in Theorem \ref{th:KK-main},
\begin{align*}
j([\rho_1][\rho_2]) & = j([\rho_1 \circ \rho_2]) 
=  j([\rho_2 \circ \rho_1])
   =  \{\widehat{\rho_2 \circ \rho_1} |_{\KK}\} \\
& = \{\hat{\rho_2}|_{\KK} \circ \hat{\rho_1} |_{\KK}\} 
= \{\hat{\rho_1} |_{\KK}\} \times  \{\hat{\rho_2} |_{\KK}\} =  j([\rho_1]) \times j([\rho_2]) 
\end{align*}
 
Next, we check the additivity of $j$, namely $j([\rho_1] \oplus [\rho_2]) = j([\rho_1]) + j([\rho_2])$.
If $s_1$ and $s_2$ are a pair isometries generating a copy of $O_2$ inside $C^*(\A)$  then a representative of 
$[\rho_1] \oplus [\rho_2]\in\Ring$ is the endomorphism $\rho=s_1 \rho_1(\cdot)s_1^* + s_2 \rho_2(\cdot)s_2^*$.
Therefore, using the fact that $\hat\rho(x) =  \hat{s}_1 \hat{\rho_1}(x) \hat{s}_1^* + \hat{s}_2 \hat{\rho_2}(x) \hat{s}_2^*$, $x \in C^*_{\locn}(\A)$, where $\hat{s_i} = \pi_{\locn}(s_i)$, $ i=1,2$, one has
\[
j([\rho_1] \oplus [\rho_2]) = \{(s_1 \rho_1(\cdot)s_1^* + s_2 \rho_2(\cdot)s_2^*)\hat{}|_{\KK}\}
= \{\hat{s}_1 \hat{\rho_1}|_{\KK} \hat{s}_1^* + \hat{s}_2 \hat{\rho_2}|_{\KK} \hat{s}_2^*\}.
\]
The subsequent Lemma \ref{lemma} and its proof (with $A=C^*_{\locn}(\A)$ and $I=\KK$) imply
\[
  T(\KK\oplus\KK)=\KK, \quad 
T(\hat{\rho_1}|_{\KK}\oplus \hat{\rho_2}|_{\KK}) T^* =\hat{s}_1 \hat{\rho_1}|_{\KK} \hat{s}_1^* + \hat{s}_2 \hat{\rho_2}|_{\KK} \hat{s}_2, 
\]
so that, according to Theorem \ref{th:KK-main}(3) and the introduction of Kasparov modules, the classes of $(\KK,\hat{\rho_1}|_{\KK},0)\oplus(\KK,\hat{\rho_1}|_{\KK},0)$ and 
$(\KK,\hat{s}_1 \hat{\rho_1}|_{\KK} \hat{s}_1^* + \hat{s}_2 \hat{\rho_2}|_{\KK} \hat{s}_2^*,0)$ in $KK(\KK,\KK)$ coincide, thus
\[
\{\hat{s}_1 \hat{\rho_1}|_{\KK} \hat{s}_1^* + \hat{s}_2 \hat{\rho_2}|_{\KK} \hat{s}_2^*\}
= \{ \hat{\rho_1}|_{\KK}\} + \{\hat{\rho_2}|_{\KK}\} = j([\rho_1]) + j([\rho_2]).
\]
Finally, the identity (\ref{rel}) follows at once from point (7) in 
Theorem \ref{th:KK-main}. 
Since $\eta$ is injective according to \cite[Theorem 4.4]{CCHW12},
  $j$ has to be injective, too.
\end{proof}
 
\begin{lemma}\label{lemma}
Let $A$ be a unital  C*-algebra containing a copy of the Cuntz algebra $O_2$ and let $I$ be a closed two-sided ideal in $A$. Then $I\oplus I \simeq I$ as right Hilbert $I$-modules.
\end{lemma}
 
\begin{proof}
Let $s_1,s_2\in A$ be a pair of isometries such that $s_1 s_1^* + s_2 s_2^* = 1$. Define the linear operator $T:I\oplus I\ra I$ by $T(a,b):= s_1 a + s_2 b$, $a,b\in I$. It is easy to see that $T$ is an isometric and surjective right $I$-module map, so that it is indeed unitary, \cite[3.2]{Lan}.
\end{proof}

\section{Some remarks}
 
First of all, Theorem \ref{th:KK-sectors} says that the left action of $\Ring$ on $K_0(\KK)$ ``factorizes" through a right action of $KK(\KK,\KK)$ and can be expressed in terms of a Kasparov product. This construction appears to depend on various special properties of the algebra $\KK$ which has been defined only for completely rational conformal nets. It is not evident how to 
generalize the definition of $\KK$  in order to find the analogues of \cite[Th.4.4]{CCHW12} and
Theorem \ref{th:KK-sectors} for a more general class of conformal nets.

Second, write $\RRing$ for the Grothendieck ring generated by the semiring $\Ring$ and  $\tilde{j}$ for the unique ring homomorphism extending $j$ to $\RRing$. It is clearly injective, owing to the preceding theorem. We would like to stress then that $\tilde{j}(\RRing)\subset KK(\KK,\KK)$ is always a \emph{proper subring} if the number of sectors $n$ is greater than $1$, \ie for non-holomorphic theories. This can be easily seen because 
$\RRing\simeq K_0(\KK)$ is commutative while $KK(\KK,\KK)$ is non-commutative. To prove the latter statement, we simply the universal coefficient theorem  \cite[V.1.5.8]{Bla06}, which holds for the algebra $\KK$ since it lies in the so-called bootstrap class  \cite[V.1.5.4]{Bla06}, and which says
\begin{equation}\label{eq:KK-Kunneth}
KK(\KK,\KK) \simeq \End(K_0(\KK))\simeq \Mat_n(\Z).
\end{equation}
 
Our last comment concerns \emph{modular invariants}. In \cite[Sec.7]{EG} it has been pointed out that modular invariants can be regarded as certain KK-classes. As we shall see now, an interpretation of modular invariants as KK-classes can be directly formulated in our setting. 
For a completely rational conformal net the braiding is always non-degenerate
\cite[Cor.37]{KLM}. Then by Rehren's construction there are invertible complex vector space  endomorphisms 
$S,  T \in \End(\C \otimes_\Z \RRing)$  giving rise to a representation of the modular group $\operatorname{SL}(2,\Z)$ on the fusion algebra $\C \otimes_\Z \RRing$, \cf \cite{Reh}.
A modular invariant is a $\Z$-module endomorphisms $Z \in \End_\Z(\RRing)$ such that: 
\begin{itemize}
\item[-] $\id \otimes_{\Z} Z \in  \End\left(\C \otimes_\Z \RRing \right)$ commutes with $S$ and $T$ (modular invariance); 
\item[-] $Z(\Ring) \subset \Ring$ (positivity);
\item[-] the matrix element $Z_{1 1}$ in the distinguished basis of sectors is $1$ 
(uniqueness of the vacuum).
\end{itemize}
Hence, as a consequence of the additive group isomorphism $\RRing \simeq K_0(\KK)$, every modular invariant $Z$ can be identified with an element  in $\End(\K_0(\KK)) \simeq KK(\KK,\KK)$. 
 
\vspace{
\baselineskip}
 
{\small
 
\noindent\textbf{Acknowledgements.}
We would like to thank Roberto Longo for stimulating comments on the relationship between the the representation theory of conformal nets and KK-theory. We also thank Martin Grensing for a very useful discussion on Kasparov products and for drawing to our attention some references.

}
 
\end{document}